\documentclass[10pt]{article}
\usepackage{amssymb,amsmath,amsthm,bm}
\usepackage[colorlinks=true, urlcolor=blue,linkcolor=blue, citecolor=blue]{hyperref}
\usepackage{caption}
\usepackage{graphicx, float}
\usepackage{color, mathtools, tikz, xcolor}
\usepackage{enumerate}
\usepackage[labelformat=simple]{subcaption}

\usetikzlibrary{decorations.pathreplacing,arrows,calc}
\usetikzlibrary{shapes.geometric}
\usetikzlibrary{decorations.markings}
\usepackage{bbm}

\usepackage[UKenglish]{babel}
\usepackage[UKenglish]{isodate}

\usepackage{geometry}
\numberwithin{equation}{section}

\numberwithin{figure}{section}
	\newtheorem{theorem}{Theorem}[section]
	
	\newtheorem{conjecture}[theorem]{Conjecture}
	
	\newtheorem{claim}[theorem]{Claim}

	\newtheorem{fact}[theorem]{Fact}

\date{\today}

\begin{document}

%\linenumbers
\title{Vertex-distinguishing  chromatic index of digraphs}

\author{Yuping Gao\footnote{School of Mathematics and Statistics, Lanzhou University, Lanzhou 730000, China.
			Email: {\tt gaoyp@lzu.edu.cn}.}
		\qquad
		Zijun Qin\footnote{School of Mathematics and Statistics, Lanzhou University, Lanzhou 730000, China. Email: {\tt 1802127189@qq.com}.}
		\qquad
		Songling Shan\footnote{Department of Mathematics and Statistics, Auburn University, Auburn, AL 36849, USA.
			Email:	{\tt szs0398@auburn.edu}.
		}
	}

\maketitle
\begin{abstract}
	
	Let $D$ be a digraph. In this note, an \emph{arc coloring} of $D$ is an assignment of colors to the arcs of $D$ such that no two arcs with a common tail receive the same color and no two arcs with a common head receive the same color.
	Under such a coloring, each vertex $v$ is associated with an \emph{out-color set} and an \emph{in-color set}, consisting of the colors assigned to the arcs with tail $v$ and to the arcs with head $v$, respectively. An arc coloring of $D$ is \emph{vertex-distinguishing} if any two distinct vertices have different out-color sets and different in-color sets.
	The minimum number of colors required for a vertex-distinguishing arc coloring of $D$ is called the \emph{vertex-distinguishing chromatic index} of $D$, denoted  $\chi_{vd}^{\prime}(D)$. In 2016, Li, Bai, He, and Sun conjectured that
	$\chi_{vd}^{\prime}(D)=k(D)$ for any digraph $D$  with at most one source and at most one sink,  where $k(D)$ is a natural lower bound determined by the outdegree and indegree sequences of $D$. We confirm this conjecture.

\medskip

\end{abstract}

\section{Introduction}

All graphs and digraphs considered in this paper are finite and simple. For two integers $p$ and $q$, let $[p,q]=\{i\in \mathbb{Z}\colon p\leq i\leq q\}$. In particular, we denote $[1,q]$ by $[q]$.
Given an integer $k\in \mathbb{N}$, an \emph{edge-$k$-coloring} of a graph $G$ is an assignment of colors from $[k]$ to the set of edges of $G$ such that no two adjacent edges receive the same color.
The \emph{chromatic index} of $G$, denoted  $\chi'(G)$, is the smallest integer $k$ for which $G$ admits an edge-$k$-coloring.
In the 1960s, Gupta~\cite{Gupta-67} and Vizing~\cite{Vizing-2-classes} independently proved the fundamental result that $\Delta(G) \le \chi'(G) \le \Delta(G)+1$.

For an \emph{edge-$k$-coloring} $\varphi$ of $G$ and $v\in V(G)$,
let $\varphi(v)=\{\varphi(e)\colon  \text{$e$ is  incident with $v$ in $G$}\}$ be the set of colors assigned to edges incident with $v$, and we call it the \emph{color-set} of $v$.
If $\varphi(u) \ne \varphi(v)$ for any two distinct vertices $u,v$ of $G$, then   $\varphi$  is called a \emph{vertex-distinguishing edge-$k$-coloring}, or a \emph{$k$-vdec} for short.   It is clear that $G$ has a $k$-vdec for some integer $k$ if and only if \emph{$G$  is vdec}, that is, $G$ has at most one isolated vertex
and  has no isolated edge.  The smallest integer $k$
for which  a vdec graph has a $k$-vdec is called  the \emph{vertex-distinguishing chromatic index} of $G$, and is denoted
by $\chi'_{vd}(G)$.
For $d\in [\delta(G), \Delta(G)]$, let $n_d$ be the total number of
vertices of degree $d$ in $G$, and let $k(G) =\min\{k\in \mathbb{N}\colon  {k \choose d} \ge n_d, d\in [\delta(G), \Delta(G)]\}$.
Then $k(G)$ is a natural lower bound on $\chi'_{vd}(G)$.
Burris and Schelp~\cite{BS1997} proposed a conjecture which is
a vertex-distinguishing analogue of the bound on the chromatic index  of a graph  by  Gupta~\cite{Gupta-67}  and  Vizing~\cite{Vizing-2-classes}.

\begin{conjecture}[Burris and Schelp~\cite{BS1997}]\label{conj:VDEC}
Let $G$ be a vdec graph. Then $\chi'_{vd}(G)\in[k(G), k(G)+1]$.
\end{conjecture}

Although Conjecture~\ref{conj:VDEC} remains open, Li, Bai, He, and Sun~\cite{LBHS2016} extended the study of vertex-distinguishing edge colorings from undirected graphs to digraphs and proposed an analogous conjecture in 2016. We first introduce the corresponding concepts for digraphs.

 Let $D=(V(D),A(D))$ be a digraph, and let $k\in\mathbb{N}$. An \emph{arc-$k$-coloring} of $D$ is a mapping $\varphi\colon A(D)\to[k]$ such that no two arcs with a common tail receive the same color and no two arcs with a common head receive the same color. Zwonek~\cite{Zwonek2006} calls such a coloring an arc coloring of the \emph{second type}. In contrast, an arc coloring of the first type requires any two arcs of the form $uv$ and $vw$ to receive distinct colors. The minimum number of colors required for an arc coloring of the second type of $D$ is denoted by $\chi'(D)$.

 For each $v\in V(D)$, let $\varphi^+(v)=\{\varphi(vw)\colon vw\in A(D)\}$ and $\varphi^-(v)=\{\varphi(wv)\colon wv\in A(D)\}$. These sets are called the \emph{out-color set} and the \emph{in-color set} of $v$, respectively. An arc coloring $\varphi$ is \emph{vertex-distinguishing} if $\varphi^+(u)\neq\varphi^+(v)$ and $\varphi^-(u)\neq\varphi^-(v)$ for every pair of distinct vertices $u,v\in V(D)$. A vertex-distinguishing arc-$k$-coloring is abbreviated as a $k$-\emph{vdac}. The minimum integer $k$ for which $D$ admits a $k$-vdac is called the \emph{vertex-distinguishing chromatic index} of $D$ and is denoted by $\chi'_{vd}(D)$.

 A \emph{source} of $D$ is a vertex with no in-arcs, while a \emph{sink} is a vertex with no out-arcs. If $D$ has two sources, then their in-color sets are both empty, and if $D$ has two sinks, then their out-color sets are both empty. Thus, a digraph admitting a vertex-distinguishing arc coloring has at most one source and at most one sink. Conversely, if $D$ has at most one source and at most one sink, then assigning a distinct color to every arc yields a vertex-distinguishing arc coloring. Therefore, $D$ admits a vertex-distinguishing arc coloring if and only if it has at most one source and at most one sink. We call such a digraph a \emph{vdac-digraph}.

For each $v\in V(D)$, let $d^+(v)$ and $d^-(v)$ denote the outdegree and indegree of $v$, respectively. Denote by $\delta^+(D)$, $\delta^-(D)$, $\Delta^+(D)$, and $\Delta^-(D)$ the minimum outdegree, minimum indegree, maximum outdegree, and maximum indegree of $D$, respectively, and let $\Delta(D)=\max\{\Delta^+(D),\Delta^-(D)\}$. For each nonnegative integer $d$, let $n_d^+(D)$ and $n_d^-(D)$ denote the numbers of vertices of outdegree $d$ and indegree $d$ in $D$, respectively.
 In analogy with $k(G)$, define
\[
k(D)=\min\left\{k\in\mathbb{N}\colon
\binom{k}{d}\geq n_d^+(D)\text{ for }d\in[\delta^+(D),\Delta^+(D)]
\text{ and }
\binom{k}{d}\geq n_d^-(D)\text{ for }d\in[\delta^-(D),\Delta^-(D)]
\right\}.
\]
It is easy to see that $\chi'_{vd}(D)\geq k(D)$. In contrast to the undirected setting, where the additional ``$+1$'' in Conjecture~\ref{conj:VDEC} is necessary for some graphs, Li, Bai, He, and Sun~\cite{LBHS2016} conjectured that the lower bound $k(D)$ is always attained for digraphs.

\begin{conjecture}[Li, Bai, He, and Sun~\cite{LBHS2016}]
	\label{conj:VDAC}
	Let $D$ be a vdac-digraph. Then $\chi'_{vd}(D)=k(D)$.
\end{conjecture}

In this paper, we confirm Conjecture~\ref{conj:VDAC}.

\begin{theorem}\label{thm:main}
	Let $D$ be a vdac-digraph. Then $\chi'_{vd}(D)=k(D)$.
\end{theorem}

\section{Proof of Theorem~\ref{thm:main}}\label{sec:proof}

Let $G$ be a graph, let $k\in\mathbb{N}$, and let
$\varphi$  be  an edge-$k$-coloring of $G$.
For two distinct colors $\alpha,\beta \in [k]$, the components of the subgraph induced by edges with colors $\alpha$ or $\beta$ are called $(\alpha, \beta)$-chains. Clearly, each $(\alpha, \beta)$-chain is either a path or an even cycle.    If the chain is
a path, we also call it  an  \emph{$(\alpha, \beta)$-path}.
If we interchange the colors $\alpha$ and $\beta$
on an $(\alpha,\beta)$-chain  of $G$, we get a new edge-$k$-coloring  of $G$. This operation is a \emph{Kempe change}.

We write $G[X,Y]$ for a bipartite graph $G$ with bipartition $(X,Y)$. To translate the arc coloring problem into an edge-coloring setting, we associate with each digraph $D$ a \emph{balanced bipartite representation} $G_D[X,Y]$, where $X=\{x_v\colon v\in V(D)\}$ and $Y=\{y_v\colon v\in V(D)\}$ are two disjoint copies of $V(D)$, and $x_u y_v\in E(G_D)$ if and only if $uv\in A(D)$.

For a bipartite graph $G[X,Y]$, let $\delta_X(G)$ and $\Delta_X(G)$ denote the minimum and maximum degrees among the vertices in $X$, respectively, and define $\delta_Y(G)$ and $\Delta_Y(G)$ analogously. Let $\Delta(G)=\max\{\Delta_X(G),\Delta_Y(G)\}$. By construction, $\Delta^+(D)=\Delta_X(G_D)$, $\Delta^-(D)=\Delta_Y(G_D)$, $\delta^+(D)=\delta_X(G_D)$, and $\delta^-(D)=\delta_Y(G_D)$.

Two arcs of $D$ have a common tail or a common head if and only if the corresponding edges of $G_D$ are adjacent. Thus, an arc coloring of $D$ is equivalent to a proper edge coloring of $G_D$. Since $\chi'(G)=\Delta(G)$ for every bipartite graph $G$ by K\"onig's  theorem~\cite{K1916}, the following fact is immediate.

\begin{fact}\label{fact:DtoB}
	Let $D$ be a digraph. Then $\chi'(D)=\chi'(G_D)=\Delta(G_D)=\Delta(D)$.
\end{fact}

An  edge coloring of a bipartite graph $G[X,Y]$ is a  \emph{partially-vertex-distinguishing} edge coloring or a \emph{pvdec} for short if (i) no two vertices in $X$ have the same color-set and (ii) no two vertices in $Y$ have the same color-set.
A bipartite graph is a \emph{pvdec-bipartite graph} if it has at most one isolated vertex from each part.
The minimum number of colors required for a pvdec of a
pvdec-bipartite graph $G$ is denoted by $\chi'_{pvd}(G)$.
For an integer $d$, let $n_{d}^X$ (resp. $n_{d}^Y$) be the number of vertices of degree $d$ in $X$ (resp. in $Y$). Define
\begin{align*}
	k_p(G) = \min \Big\{ k \in \mathbb{N}\colon  \binom{k}{d}\ge n_{d}^X \text{ for } d\in[\delta_X(G), \Delta_X(G)],\ \binom{k}{d}\ge n_{d}^Y \text{ for } d\in[\delta_Y(G),\Delta_Y(G)]\Big\}.
\end{align*}

Under the above correspondence, an arc coloring $\varphi$ of $D$ is
vertex-distinguishing if and only if the corresponding edge coloring of
$G_D$ is partially vertex-distinguishing. Indeed, the color set of $x_v$
is $\varphi^+(v)$, while the color set of $y_v$ is $\varphi^-(v)$.
Moreover, $k_p(G_D)=k(D)$. Therefore, Theorem~\ref{thm:main} follows
from the following stronger statement, which we prove.

\begin{theorem}\label{thm}
	Let $G$ be a pvdec-bipartite graph. Then $\chi'_{pvd}(G)=k_p(G)$.
\end{theorem}

\begin{proof}
	Every $k$-pvdec of $G$ assigns distinct $d$-element color-sets to the
	vertices of degree $d$ in each part. Hence
	$\binom{k}{d}\geq n_d^X$ and $\binom{k}{d}\geq n_d^Y$ for every relevant
	$d$, and therefore $\chi'_{pvd}(G)\geq k_p(G)$.

	We prove the reverse inequality. Let $k=k_p(G)$. By the definition of
	$k_p(G)$, we have $k\geq\Delta(G)$. Since $G$ is bipartite, K\"onig's
	theorem~\cite{K1916}  implies that $G$ admits an edge-$k$-coloring.
	
	For an  edge-$k$-coloring $f$ of $G$, $Z\in\{X,Y\}$ and $S\subseteq[k]$, let
	\[
	n_S^Z(f)=|\{v\in Z\colon f(v)=S\}|.
	\]
	We call $S$ a \emph{bad color set} in $Z$ if $n_S^Z(f)\geq2$, and a
	\emph{hole} in $Z$ if $n_S^Z(f)=0$.
	
	Among all  edge-$k$-colorings $f$  of $G$, choose one minimizing
	\[
	\Phi(f)=\sum_{S\subseteq[k]}
	\left((n_S^X(f))^2+(n_S^Y(f))^2\right).
	\]
	We call such a coloring \emph{optimal}. We shall prove that an optimal
	coloring has no bad color set.
	
	Fix two distinct colors $\alpha,\beta\in[k]$. For a set
	$S\subseteq[k]$, let $i_{\alpha\beta}(S)$ be the set obtained from $S$
	by interchanging $\alpha$ and $\beta$, and write
	$S^*=i_{\alpha\beta}(S)$. Let
	\[
	[\alpha\diamond\beta]
	=\{S\subseteq[k]\colon |S\cap\{\alpha,\beta\}|=1\}.
	\]
	Thus, $S^*\neq S$ for every $S\in[\alpha\diamond\beta]$.
	
	We define an auxiliary multigraph $H=H(f,\alpha,\beta)$. Its vertex
	set is
	\[
	V(H)=\{v\in V(G)\colon |f(v)\cap\{\alpha,\beta\}|=1\}.
	\]
	Its edges are of the following two types.
	
	\begin{itemize}
		\item[(i)] For every  $(\alpha,\beta)$-path in $G$, add a
		\emph{Kempe edge} joining its two endvertices.
		
		\item[(ii)] For each $Z\in\{X,Y\}$ and each unordered orbit
		$\{S,S^*\}$ of $i_{\alpha\beta}$ on
		$[\alpha\diamond\beta]$, choose a matching of size
		\[
		\min\{n_S^Z(f),n_{S^*}^Z(f)\}
		\]
		between the vertices of $Z$ with color set $S$ and the vertices of
		$Z$ with color set $S^*$, and add its edges as \emph{matching edges}.
	\end{itemize}
	
	Since $|f(v)\cap\{\alpha,\beta\}|=1$ for every $v\in V(H)$, the
	vertex $v$ has degree one in the subgraph of $G$ induced by the edges
	colored $\alpha$ or $\beta$. Hence $v$ is an endvertex of a unique
	  $(\alpha,\beta)$-chain  and is incident with exactly one
	Kempe edge of $H$. Moreover, $v$ is incident with at most one matching
	edge. It follows that every component of $H$ is a path or a cycle, and
	its edges alternate between Kempe edges and matching edges. The
	endvertices of the path components of $H$ are precisely the vertices
	not covered by the chosen matching edges.
	
	Let $P$ be a path component of $H$. The \emph{Kempe exchange along
		$P$} consists of interchanging $\alpha$ and $\beta$ on every
	$(\alpha,\beta)$-path whose Kempe edge belongs to $P$. Let $f'$
	be the resulting   edge-$k$-coloring. Since each vertex of $P$
	is an endvertex of exactly one of the  $(\alpha,\beta)$-paths involved, we have
	\begin{equation}\label{eq:flip}
		f'(v)=
		\begin{cases}
			i_{\alpha\beta}(f(v)), & v\in V(P),\\
			f(v), & v\notin V(P).
		\end{cases}
	\end{equation}
	
	For every matching edge of $P$, its endvertices lie in the same part
	of $G$ and have color sets $R$ and $R^*$ for some
	$R\in[\alpha\diamond\beta]$. By~\eqref{eq:flip}, these two vertices
	exchange their color sets, so their contributions to the color-set
	multiplicities cancel. Since every vertex of $P$ other than its two
	endvertices is incident with a matching edge, the multiplicities
	change only through the two endvertices of $P$.
	
	More precisely, let $u$ and $w$ be the endvertices of $P$, and let
	$Z_u$ and $Z_w$ be the parts containing $u$ and $w$, respectively.
	The exchange decreases $n_{f(u)}^{Z_u}$ and $n_{f(w)}^{Z_w}$ by one
	and increases $n_{f(u)^*}^{Z_u}$ and $n_{f(w)^*}^{Z_w}$ by one, with the changes added together whenever two or more of these terms refer to the same entry.
	We first establish a parity property of the endvertices of $P$.

	\begin{claim}\label{claim:parity-property}
		Let $P$ be a path component of $H(f,\alpha,\beta)$ with endvertices $u$
		and $w$. Then
		\begin{enumerate}[{\rm(i)}]
			\item $f(u)\cap\{\alpha,\beta\}\neq f(w)\cap\{\alpha,\beta\}$ if  $u$ and $w$
			lie in the same part of $G$; and
			\item $f(u)\cap\{\alpha,\beta\}=f(w)\cap\{\alpha,\beta\}$	if  $u$ and $w$   lie in distinct parts of $G$.
		\end{enumerate}

	\end{claim}

	\begin{proof}[Proof of Claim~\ref{claim:parity-property}]
	Write $P=v_0v_1\ldots v_{2r+1}$, where $v_0=u$ and
	$v_{2r+1}=w$. Since the edges of $P$ alternate between Kempe edges
	and matching edges, beginning and ending with Kempe edges, $P$ has
	odd length.
	
	For $v\in V(P)$, define
	\[
	\mu(v)=
	\begin{cases}
		0, & \alpha\in f(v),\\
		1, & \beta\in f(v),
	\end{cases}
	\qquad
	\sigma(v)=
	\begin{cases}
		0, & v\in X,\\
		1, & v\in Y.
	\end{cases}
	\]
	We claim that, for every $i\in[0,2r]$,
	\begin{equation}\label{eq:edge-parity}
		\mu(v_i)+\mu(v_{i+1})
		\equiv
		1+\sigma(v_i)+\sigma(v_{i+1})
		\pmod 2.
	\end{equation}
	
	If $v_iv_{i+1}$ is a matching edge, then $v_i$ and $v_{i+1}$ lie in
	the same part and their color sets are interchanged by
	$i_{\alpha\beta}$. Thus, their $\sigma$-values are equal and their
	$\mu$-values are different, so~\eqref{eq:edge-parity} holds.
	
	Suppose that $v_iv_{i+1}$ is a Kempe edge corresponding to an
	$(\alpha,\beta)$-path $Q$. If $Q$ has odd length, then its
	endvertices lie in distinct parts and are incident with the same color
	from $\{\alpha,\beta\}$. If $Q$ has even length, then its endvertices
	lie in the same part and are incident with different colors from
	$\{\alpha,\beta\}$. Hence~\eqref{eq:edge-parity} also holds for a
	Kempe edge.
	
	Summing~\eqref{eq:edge-parity} over all edges of $P$, the
	contributions of the internal vertices cancel modulo two, and we
	obtain
	\[
	\mu(u)+\mu(w)
	\equiv
	1+\sigma(u)+\sigma(w)
	\pmod 2.
	\]
	The claim follows.
	\end{proof}

	\medskip
	We next prove a balance property of an optimal coloring.
	
	\begin{claim}\label{claim:balance-property}
 	For every $S\in[\alpha\diamond\beta]$, we have
 \[
 |n_S^X(f)-n_{S^*}^X(f)|\leq1
 \qquad\text{and}\qquad
 |n_S^Y(f)-n_{S^*}^Y(f)|\leq1.
 \]
\end{claim}

\begin{proof}[Proof of Claim~\ref{claim:balance-property}]
By symmetry, it suffices to prove
$n_S^X(f)\leq n_{S^*}^X(f)+1$. Suppose instead that
$n_S^X(f)\geq n_{S^*}^X(f)+2$.
The chosen matching between the two classes leaves at least two
vertices of $X$ with color set $S$ unmatched. Let $u$ be one of them,
let $P$ be the path component of $H$ containing $u$, and let $w$ be
the other endvertex of $P$. Let $Z\in\{X,Y\}$ be the part containing
$w$.

We first show that the changes at $u$ and $w$ affect distinct entries in the definition of $\Phi$. Suppose first that $w\in X$. Then $f(w)\notin\{S,S^*\}$. Indeed, Claim~\ref{claim:parity-property} gives $f(w)\neq S$. Moreover, since $n_S^X(f)>n_{S^*}^X(f)$, every vertex of $X$ with color set $S^*$ is covered by a matching edge, whereas $w$ is unmatched. Hence $f(w)\neq S^*$. It follows also that $f(w)^*\notin\{S,S^*\}$. Thus, the entries affected at $w$ are distinct from the entries $n_S^X$ and $n_{S^*}^X$ affected at $u$. If $w\in Y$, then the entries affected at $u$ and $w$ correspond to different parts and are therefore distinct.

Let $f'$ be obtained from $f$ by the Kempe exchange along $P$. Then the contributions of the two endvertices to $\Phi(f')-\Phi(f)$ can be computed separately. At
$u$, one occurrence is transferred from $S$ to $S^*$, contributing
\begin{align*}
	&\big(n_S^X(f)-1\big)^2+\big(n_{S^*}^X(f)+1\big)^2
	-\big(n_S^X(f)\big)^2-\big(n_{S^*}^X(f)\big)^2\\
	&\hspace{25mm}
	=2\big(n_{S^*}^X(f)-n_S^X(f)\big)+2
	\leq-2
\end{align*}
to $\Phi(f')-\Phi(f)$.

Write $R=f(w)$. Since $w$ is unmatched, $R$ is a majority class in
the pair $\{R,R^*\}$ in the part $Z$. Hence
$n_R^Z(f)\geq n_{R^*}^Z(f)+1$, and the contribution at $w$ is
\[
2\big(n_{R^*}^Z(f)-n_R^Z(f)\big)+2\leq0.
\]
Consequently, $\Phi(f')<\Phi(f)$, contradicting the optimality of
$f$. This proves Claim~\ref{claim:balance-property}.
\end{proof}

	\medskip
	We shall also use the following consequence.

	\begin{claim}\label{claim:multi-change}
	Suppose that $S\in[\alpha\diamond\beta]$ satisfies
	$n_S^X(f)=n_{S^*}^X(f)+1$.
	Let $u$ be the unique unmatched vertex of $X$ with $f(u)=S$, let $P$
	be the path component of $H(f,\alpha,\beta)$ containing $u$, and let
	$w$ be the other endvertex of $P$. Then the coloring $f'$ obtained by
	the Kempe exchange along $P$ is optimal. Moreover, if $Z$ is the part
	containing $w$, then the color-set multiplicities change by
	transferring one occurrence from $S$ to $S^*$ in $X$ and one
	occurrence from $f(w)$ to $f(w)^*$ in $Z$.
	\end{claim}

 \begin{proof}[Proof of Claim~\ref{claim:multi-change}]
 	By the same argument as in the proof of Claim~\ref{claim:balance-property},
 the changes at $u$ and $w$ affect distinct entries in the definition of $\Phi$.
 The contribution from the pair $\{S,S^*\}$ in $X$ is $2\big(n_{S^*}^X(f)-n_S^X(f)\big)+2=0$.
 Write $R=f(w)$. Since $w$ is unmatched, we have
 $n_R^Z(f)\geq n_{R^*}^Z(f)+1$. By Claim~\ref{claim:balance-property}, equality holds. Hence the
 contribution from the pair $\{R,R^*\}$ in $Z$ is also zero. Therefore,
 $\Phi(f')=\Phi(f)$, so $f'$ is optimal. The assertion concerning the
 multiplicities follows from the endpoint bookkeeping above.
 \end{proof}

	We now complete the proof of Theorem~\ref{thm}. Suppose, to the contrary, that every
	optimal edge-$k$-coloring has a bad color set.
	Let $J(k,d)$ be the Johnson graph whose vertices are the $d$-subsets of
	$[k]$, where two sets are adjacent if one can be obtained from the other
	by replacing one color. The graph $J(k,d)$ is connected, since any
	$d$-subset can be transformed into any other $d$-subset by replacing, one
	at a time, the colors that belong to the first set but not to the second.

	Among all optimal  edge-$k$-colorings $f$, all parts
	$Z\in\{X,Y\}$, all integers $d$, all bad color sets
	$S\in\binom{[k]}{d}$ in $Z$, and all holes
	$T\in\binom{[k]}{d}$ in $Z$, choose $(f,Z,d,S,T)$ so that
	\[
	\ell=\operatorname{dist}_{J(k,d)}(S,T)
	\]
	is minimum, where $\operatorname{dist}_{J(k,d)}(S,T)=|S\setminus T|$. Such a hole exists because
	$n_d^Z\leq\binom{k}{d}$, while some $d$-element color-set occurs at
	least twice in $Z$. By symmetry, we may assume that $Z=X$.
	
	We have $\ell\geq2$. Indeed, if $\ell=1$, then there are distinct
	colors $\alpha,\beta$ such that $T=S^*$. Since $S$ is bad and $T$ is
	a hole, we have
	$n_S^X(f)\geq2$ and $n_{S^*}^X(f)=0$, contradicting Claim~\ref{claim:balance-property}.
	
	Let $Q_0Q_1\ldots Q_\ell$ be a shortest path in $J(k,d)$ from
	$Q_0=T$ to $Q_\ell=S$. Since
	$\operatorname{dist}_{J(k,d)}(Q_1,Q_\ell)=\ell-1$, the first step
	must remove an element of $Q_0\setminus Q_\ell$ and add an element of
	$Q_\ell\setminus Q_0$. Hence there exist
	$\alpha\in Q_0\setminus Q_\ell$ and
	$\beta\in Q_\ell\setminus Q_0$ such that $Q_1=(Q_0\setminus\{\alpha\})\cup\{\beta\}$.
	Consequently, $Q_0=Q_1^*$ and
	\begin{equation}\label{eq:trace}
		Q_1\cap\{\alpha,\beta\}
		=
		Q_\ell\cap\{\alpha,\beta\}
		=
		\{\beta\}.
	\end{equation}
	
	The set $Q_1$ is not a hole, since otherwise it would be a hole at
	distance $\ell-1$ from $Q_\ell$, contradicting the choice of
	$\ell$. By Claim~\ref{claim:balance-property}, $n_{Q_1}^X(f)\leq n_{Q_0}^X(f)+1=1$.
	Hence $n_{Q_1}^X(f)=1$.
	
	Let $u$ be the unique vertex of $X$ with $f(u)=Q_1$. Since
	$n_{Q_1}^X(f)=n_{Q_0}^X(f)+1$, the vertex $u$ is the unique unmatched
	vertex of $X$ with color set $Q_1$ in
	$H(f,\alpha,\beta)$. Let $P$ be the path component of $H$ containing
	$u$, and let $w$ be its other endvertex. Let $f'$ be obtained from
	$f$ by the Kempe exchange along $P$. By Claim~\ref{claim:multi-change}, $f'$ is optimal and $n_{Q_1}^X(f')=0, n_{Q_0}^X(f')=1$.
	In particular, $Q_1$ is a hole in $X$ under $f'$.
	
	If $Q_\ell$ remains bad in $X$ under $f'$, then $Q_1$ is a hole at
	distance $\ell-1$ from $Q_\ell$, contradicting the minimality of
	$\ell$. Therefore, $n_{Q_\ell}^X(f')\leq1$.
	On the other hand, since $Q_\ell$ is bad in $X$ under $f$,
$n_{Q_\ell}^X(f)\geq2$.
	Consequently, $n_{Q_\ell}^X(f')<n_{Q_\ell}^X(f)$.
	Since $\ell\geq2$, we have
	$Q_\ell\notin\{Q_0,Q_1\}$. Thus, the change at $u$ from $Q_1$ to
	$Q_0$ does not affect $n_{Q_\ell}^X$. By
	Claim~\ref{claim:multi-change}, the decrease in $n_{Q_\ell}^X$ must
	therefore be caused by the other endvertex $w$ of $P$. It follows that
	$w\in X$ and $f(w)=Q_\ell$. Moreover, the change at $w$ decreases this
	multiplicity by exactly one. Hence
	\[
	n_{Q_\ell}^X(f)=2,
	\qquad
	n_{Q_\ell}^X(f')=1,
	\qquad
	w\in X,
	\qquad
	f(w)=Q_\ell.
	\]
	
	Thus, the two endvertices $u$ and $w$ of $P$ both lie in $X$. By
	Claim~\ref{claim:parity-property}, $	f(u)\cap\{\alpha,\beta\}
	\neq f(w)\cap\{\alpha,\beta\}$.
	However, $f(u)=Q_1$ and $f(w)=Q_\ell$, contradicting
	\eqref{eq:trace}.
	
	This contradiction shows that some optimal   edge-$k$-coloring
	has no bad color set in either part. Hence no two vertices in $X$
	have the same color set and no two vertices in $Y$ have the same
	color set. Therefore, this coloring is a $k$-pvdec of $G$, and
	$\chi'_{pvd}(G)\leq k=k_p(G)$. Together with the lower bound, we
	obtain $\chi'_{pvd}(G)=k_p(G)$.
\end{proof}

\section*{Acknowledgements}

We thank Weihua He of Guangdong University of Technology for bringing this problem to our attention.

\bibliographystyle{abbrv}
\bibliography{vdac}

\end{document}